\documentclass{amsart}

\usepackage{amsmath}
\usepackage{amsthm}
\usepackage{hyperref}
\usepackage{amsfonts,graphics,amsthm,amsfonts,amscd,latexsym}
\usepackage{epsfig}
\usepackage{flafter}
\usepackage{mathtools}
\usepackage{comment}
\usepackage{stmaryrd}

\usepackage{mathabx,epsfig}
\def\acts{\mathrel{\reflectbox{$\righttoleftarrow$}}}
\def\actsr{\mathrel{\righttoleftarrow}}

\hypersetup{
    colorlinks=true,    
    linkcolor=blue,          
    citecolor=blue,      
    filecolor=blue,      
    urlcolor=blue           
}
\usepackage{tikz}
\usetikzlibrary{graphs,positioning,arrows,shapes.misc,decorations.pathmorphing}

\tikzset{
    >=stealth,
    every picture/.style={thick},
    graphs/every graph/.style={empty nodes},
}

\tikzstyle{vertex}=[
    draw,
    circle,
    fill=black,
    inner sep=1pt,
    minimum width=5pt,
]
\usepackage[position=top]{subfig}
\usepackage[alphabetic,backrefs]{amsrefs}
\usepackage{amssymb}
\usepackage{color}

\calclayout

\usetikzlibrary{decorations.pathmorphing}
\tikzstyle{printersafe}=[decoration={snake,amplitude=0pt}]

\renewcommand{\qq}{\mathbb{Q}}
\newcommand{\zz}{\mathbb{Z}}

\newcommand{\kk}{\mathbb{K}}

\def\O#1.{\mathcal {O}_{#1}}			
\def\pr #1.{\mathbb P^{#1}}				
\def\af #1.{\mathbb A^{#1}}			
\def\ses#1.#2.#3.{0\to #1\to #2\to #3 \to 0}	
\def\xrar#1.{\xrightarrow{#1}}			
\def\K#1.{K_{#1}}						
\def\bA#1.{\mathbf{A}_{#1}}			
\def\bM#1.{\mathbf{M}_{#1}}				
\def\bL#1.{\mathbf{L}_{#1}}				
\def\bB#1.{\mathbf{B}_{#1}}				
\def\bK#1.{\mathbf{K}_{#1}}			
\def\subs#1.{_{#1}}					
\def\sups#1.{^{#1}}

\usepackage{tikz}
\usetikzlibrary{matrix,arrows,decorations.pathmorphing}

  \newtheorem{introthm}{Theorem}

  \newtheorem{introcor}{Corollary}

  \newtheorem{theorem}{Theorem}[section]
  \newtheorem{lemma}[theorem]{Lemma}
  \newtheorem{proposition}[theorem]{Proposition}

  \newtheorem{notation}[theorem]{Notation}
  \newtheorem{definition}[theorem]{Definition}

\newtheorem{remark}[theorem]{Remark}

\theoremstyle{remark}

\numberwithin{equation}{section}

\usepackage[all]{xy}

\begin{document}

\title[Small quotient minimal log discrepancies]{Small quotient minimal log discrepancies}

\author[J.~Moraga]{Joaqu\'in Moraga}
\address{Department of Mathematics, Princeton University, Fine Hall, Washington Road, Princeton, NJ 08544-1000, USA
}
\email{jmoraga@princeton.edu}

\maketitle

\begin{abstract}
We prove that for each positive integer $n$ there exists a positive number $\epsilon_n$ so that $n$-dimensional toric quotient singularities satisfy the ACC for mld's on the interval $(0,\epsilon_n)$.
In the course of the proof, we will show a geometric Jordan property for finite automorphism groups of affine toric varieties. 
\end{abstract}

\setcounter{tocdepth}{1} 
\tableofcontents

\section{Introduction}

Given a Kawamata log terminal singularity, we can attach to it an invariant called the 
minimal log discrepancy, or mld for short~\cite{Kol13}. This invariant, although easy to define, is really deep and its understanding is closely related to the termination of flips~\cite{Sho04}.
In this direction, one of the most challenging problems is to prove that the minimal log discrepancies on a fixed dimension satisfy the ascending chain condition, whenever the coefficients of the boundary belongs to a set with the descending chain condition.
Furthermore, it is expected that accumulation points come from lower-dimensional minimal log discrepancies.
This conjecture is known as the ACC for mld's.
Alexeev and Shokurov proved the case of klt surface singularities~\cites{Sho91,Ale93}.
Borisov and Ambro proved the case of toric pairs~\cites{Bor97,Amb06}.
For a fixed germ and boundaries with coefficients on a finite set, this is proved by Kawakita~\cite{Kaw14}.
Nakamura proved the case of three dimensional canonical singularitites~\cite{Nak16}.
The case of a smooth point with a boundary is open in dimension at least three~\cite{MN18}.
The case of exceptional singularities was understood in~\cites{Mor18a,HLS19}.
Recently, Mallory proved the case of determinantal singularities~\cite{Mal19}.
Even if some progress has been made towards this conjecture, 
the case of quotient singularities is still open. 
At the same time, it's been understood that small minimal log discrepancies 
are better behaved~\cite{Liu18}.
Recently, some interesting results have been obtained for minimal log discrepancies
of $3$-folds around one~\cites{Jia19,LX19}.
In this article, we give a positive answer to this conjecture in the case
of small quotient minimal log discrepancies.

\begin{introthm}\label{introthm:small-quotient-mlds}
Let $n$ be a positive integer.
There exists a positive integer $\epsilon_n$ satisfying the following.
Consider the set
\[
 {\rm Qmld}_n:=\{ {\rm mld}_x(X) \mid \text{ $x\in X$ is a quotient singularity and $\dim(X)=n$ }\}.
\]
Then, the set 
\[
{\rm Qmld}_n \cap (0,\epsilon_n) 
\]
satisfies the ascending chain condition.
Furthermore, the accumulation points of ${\rm Qlmd}_n\cap (0,\epsilon_n)$
are minimal log discrepancies of quotient singularities of dimension at most $n-1$
with a boundary with standard coefficients.
\end{introthm}

In the course of the proof, it will be clear that our sequence of real numbers $\epsilon_n$ converges to zero when $n$ goes to infinite.
We expect that $\epsilon_n^{-1}$ behaves like
$n!$ asymptotically.

The proof of our main theorem has its roots in the study of the local topology of klt singularities.
It is well-known that the regional fundamental group of a klt singularity is finite~\cites{Xu14,TX17,Bra20}.
In~\cite{BFMS20}, the authors proved that the regional fundamental group of a $n$-dimensional klt singularity is \textit{almost} abelian of rank at most $n$, i.e., 
it contains an abelian normal subgroup of rank at most $n$ and index bounded by $O_n(1)$.
The study of the regional fundamental group has important implications on the geometry of klt singularities.
For instance, the author proved that $n$-dimensional klt singularities 
with large abelian subgroups of rank $n$ in their  regional fundamental group
are close to being degenerations of toric quotient singularities~\cites{Mor20a,Mor20b}.
In this article, we show how the Jordan property can be used to understand minimal log discrepancies.

Our first aim is to make an improvement of the Jordan property for ${\rm GL}_n(\kk)$
to automorphisms groups of affine toric varieties.
Given an affine variety $X$ and a point $x\in X$,
we denote by $\widehat{X}$ the local completion of $X$ at $x$.
We denote by ${\rm Aut}(X,x)$ (resp. ${\rm Aut}(\widehat{X},x)$)
the automorphism group of $X$ (resp. $\widehat{X}$) which fixes $x$.
We prove the following {\it geometric Jordan property} for affine toric singularities.

\begin{introthm}\label{introthm:affine-geom-jordan}
Let $n$ be a positive integer.
There exists a constant $N(n)$, only depending on $n$,
which satisfies the following.
Let $x\in X$ be a $n$-dimensional affine toric variety
and let $G<{\rm Aut}(X,x)$ be a finite subgroup.
Then, there exists a normal abelian subgroup $A\leqslant G$ of index at most $N(n)$
and rank at most $n$, so that 
$A<\mathbb{G}_m^n \leqslant {\rm Aut}(\widehat{X},x)$
where $\mathbb{G}_m^n$ is a maximal torus in ${\rm Aut}(\widehat{X},x)$.
\end{introthm}

If $x\in X$ is a smooth point,
the above theorem follows from:
Luna's slice theorem, 
the Jordan property for finite subgroups of ${\rm GL}_n(\kk)$, 
and the diagonalization of abelian actions.
By~\cite{BFMS20}, we know that any finite group $G$ acting on a klt germ
will contain an abelian normal subgroup $A$ of index bounded by $O_n(1)$.
However, the new geometric part of the statement is that,
in the germ of the singularity, we can realize $A$ as acting as a subgroup of a maximal torus.
To prove the above theorem, 
we will translate some results of~\cite{SMS19}
about the automorphisms group of complete toric varieties
to the local setting. 
Then, we will use Cox rings~\cite{ADHL15} to reduce the problem 
to study weighted versions of ${\rm GL}_n(\kk)$.
As a corollary, we obtain the following projective statement.

\begin{introcor}\label{introcor:proj-geom-jordan}
Let $n$ be a positive integer.
There exists a constant $N(n)$, only depending on $n$,
which satisfies the following.
Let $X$ be a $n$-dimensional projective toric variety
and let $G<{\rm Aut}(X)$ be a finite subgroup.
Then, there exists a normal abelian subgroup $A\leqslant G$
of index at most $N(n)$ and rank at most $n$, 
so that $A<\mathbb{G}_m^n \leqslant {\rm Aut}(X)$
where $\mathbb{G}_m^n$ is a maximal torus of ${\rm Aut}(X)$.
\end{introcor}

Using Theorem~\ref{introthm:affine-geom-jordan}, and studying 
how finite automorphisms lift to toric blow-ups, 
we prove the following factorization theorem of toric quotient singularities.

\begin{introthm}\label{introthm:quotient-diagram}
Let $n$ be a positive integer.
There exists a constant $\epsilon_n$, only depending on $n$,
which satisfies the following.
Let $x\in (X,B)$ be a $n$-dimensional toric quotient pair.
Assume that ${\rm mld}_x(X,B)\in (0,\epsilon_n)$.
Then, there exist two quotient morphisms of log pairs
\[
z\in Z\rightarrow y\in (Y,B_Y)\rightarrow x\in (X,B_X)
\]
satisfying the following conditions: 
\begin{enumerate}
\item $Z\rightarrow Y$ and $Y\rightarrow X$ are Galois quotients,
\item $Z$ is an affine toric variety,
\item $(Y,B_Y)$ is formally toric at $y$, 
\item the degree of $Y\rightarrow X$ is bounded by $\epsilon_n^{-1}$, and
\item ${\rm mld}_y(Y,B_Y)={\rm mld}_x(X,B_X)$,
\end{enumerate}
In particular, in the interval $(0,\epsilon_n)$, all 
toric quotient minimal log discrepancies
are toric minimal log discrepancies.
\end{introthm}

See Definition~\ref{def:formally} for the definition of formally toric pairs
and Definition~\ref{def:toric-quotient} for the definition of toric quotient pairs.
In the case that $x\in X$ is a quotient singularity,
we can always assume that $Z$ is the affine space.
In the above theorem, $(X,B)$ is the log quotient of 
an affine toric variety by a finite group.
In particular, the coefficients of $B$ are standard.
This theorem also proves that in the interval
$(0,\epsilon_n)$ all quotient minimal log discrepancies
are also toric minimal log discrepancies.
Hence, Theorem~\ref{introthm:small-quotient-mlds}
is a direct consequence of Theorem~\ref{introthm:quotient-diagram}.
The main difficulty on proving the above theorem
is controlling the equality ${\rm mld}_y(Y,B_Y)={\rm mld}_x(X,B_X)$.
In general, the minimal log discrepancy can change even
if the Galois cover is an involution.
Using~\cite{Amb09}
and the above theorem, we will obtain the following corollary.

\begin{introcor}\label{introcor:index}
Let $n$ be a positive integer.
There exists a constant $\epsilon_n$, only depending on $n$,
which satisfies the following.
Let $x\in (X,B)$ be a toric quotient pair of dimension $n$
with $m:={\rm mld}_x(X,B)\in (0,\epsilon_n)$.
Then, the Cartier index of $K_X+B$, formally at $x\in X$, 
belongs to a finite set
$\mathcal{F}(m,n)$ which only depends on $m$ and $n$.
\end{introcor}

The above Corollary
gives a positive answer to a conjecture of Shokurov in the case of small quotient minimal log discrepancies.

\subsection*{Acknowledgements} 
The author would like to thanks
Alvaro Liendo,
Antonio Laface,
Florin Ambro, and
Ivan Arzhantsev
for many useful comments.

\section{Preliminaries}

In this section, 
we recall the preliminaries that 
will be used in this article:
Toric geometry, 
toric minimal log discrepancies, 
and automorphisms of toric varieties.
Throughout this article, we work over an algebraically closed field $\mathbb{K}$ of characteristic zero.
We denote by $\mathbb{G}_m^k$ the $k$-dimensional $\mathbb{K}$-torus.
We require Galois morphisms to be finite and surjective.
However, we do not assume Galois morphism to be \'etale.
The rank of a finite group refers to the minimal number of generators.
We will use some standard results about toric geometry~\cites{Ful93,CLS11,Cox95},
linear algebraic groups~\cites{Hum75,Bor91}, and
the minimal model program~\cites{KM98,BCHM10}

\subsection{Toric geometry} 
In this subsection, we recall
some statements about toric geometry, toric blow-ups
and  minimal log discrepancies.

\begin{definition}
{\em
A {\em toric variety} is a normal variety $X$ with an 
open subset isomorphic to $\mathbb{G}_m^n$, 
so that the action of $\mathbb{G}_m^n$ on itself by multiplication extends to all of $X$.
Equivalently, $X$ is a normal variety of dimension $n$ endowed with an effective actionof the $n$-dimensional $\mathbb{K}$-torus.
}
\end{definition}

\begin{definition}
{\em 
Let $N$ be a free finitely generated abelian group and let $M={\rm Hom}(N,\zz)$ be its dual.
We denote by $N_\qq$ (resp $M_\qq$) the associated
$\qq$-vector space $N_\qq:=N\otimes_\zz \qq$ (resp. $M_\qq:=M\otimes_\zz \qq$).
Let $\sigma \subset N_\qq$ be a pointed convex polyhedral cone.
We denote by $\sigma^\vee \subset M_\qq$ its dual.
We write $\kk[\sigma^\vee \cap M]$ the monoid ring of the monoid $(\sigma^\vee,+)$.
The variety $X_\sigma:= {\rm Spec}(\kk[\sigma^\vee \cap M])$ is an affine variety
endowed with an action of ${\rm Spec}(\kk[M])\simeq \mathbb{G}_m^n$.
We say that $X(\sigma)$ is an {\em affine toric variety}.
It is well known that every affine toric variety arises from the choice
of a pointed polyhedral cone in a rational vector space (see, e.g.,~\cite{CLS11}).
}
\end{definition}

\begin{definition}
{\em
A {\em fan of cones} $\Sigma$ on $N_\qq$ is a finite set of pointed convex polyhedral cones
in $N_\qq$ satisfying the following conditions:
\begin{itemize}
\item If $\sigma \in \Sigma$ and $\tau\leqslant \sigma$ is a face, then $\tau \in \Sigma$, 
\item If $\sigma,\sigma' \in \Sigma$, then $\sigma \cap \sigma'\in \Sigma$ is a face of 
both $\sigma$ and $\sigma'$.
\end{itemize}
Given a fan of cones $\Sigma$ on $N_\qq$, we can construct a toric variety as follows:
For each cone $\sigma \in \Sigma$, we can consider the affine toric variety $X(\sigma)$.
If $\tau\leqslant \sigma$ is a face, then $X(\tau)\subset X(\sigma)$ is an open $\mathbb{G}_m^n$-invariant subvariety.
For each two cones $\sigma,\sigma'\in \Sigma$, we can glue the affine toric varieties
$X(\sigma)$ and $X(\sigma')$ along $X(\sigma'\cap\sigma)$.
This process gives us a toric variety which we denote $X(\Sigma)$ and call it
the {\em toric variety associated to the fan} $\Sigma$ in $N_\qq$.
It is well known that every toric variety arises from the choice
of a fan in a rational vector space (see, e.g.,~\cite{CLS11}).
}
\end{definition}

\begin{notation}
{\em
Given a variety $X$, we may endow it with different structures of toric variety.
For instance, if $X$ is toric, any maximal torus of ${\rm Aut}^0(X)$ will give $X$ the structure of a toric variety.
Although, all such different structures differ by automorphisms of $X$.
To avoid ambiguity, given a toric variety $X$ we will denote by $\mathbb{T}_X$ a fixed
torus acting on it, i.e., we will fix a choice of a maximal torus of ${\rm Aut}^0(X)$.
We will say that a morphism from (or to) $X$ is toric if it is toric with respect to $\mathbb{T}_X$.
For instance, in the case of an affine toric variety $X$, fixing $\mathbb{T}_X$ is equivalent
to fixing a choice of $\sigma$ in a rational $\qq$-vector space $N_\qq$.
}
\end{notation}

\begin{definition}\label{def:formally}{\em 
Let $X$ be an affine variety and $x\in X$ be a closed point.
We write $X={\rm Spec}(R)$ and $m\subset R$ for the maximal ideal of $x$.
We denote by $\widehat{X}_x$ the local completion of $X$ at $x$,
i.e., the spectrum of the completion of the localization of $R$ at $m$.

Let $x\in X$ be a closed point on an algebraic variety.
We say that $X$ is {\rm formally toric} at $x$
if there exists an affine toric variety $T$ so that
$\widehat{X}_x\simeq \widehat{T}_{t}$,
where $x_0\in X_0$ is a torus invariant point.
Let $B$ be an effective divisor on $X$ through $x$.
We say that $(X,B)$ is a {\rm formally toric pair}
if the above isomorphism identifies $B$ with
a toric boundary of $T$.
}
\end{definition} 

\subsection{Cox rings}

In this subsection, we recall the Cox ring of a toric variety.

\begin{definition}{\em  
Let $X$ be an algebraic variety which is irreducible, normal, with only constant invertible functions,
and finitely generated Class group.
We define its {\em Cox ring} to be
\[
{\rm Cox}(X):=\bigoplus_{[D]\in {\rm Cl}(X)}H^0(X,\mathcal{O}_X(D)).
\]
If the Cox ring is finitely generated over $\kk$, then we say that
$X$ is a {\em Mori dream space}.
}
\end{definition}

\begin{definition}{\em 
We say that a fan $\Sigma$ in $N_\qq$ is {\em degenerate}
if the support of the fan is not contained in a proper $\qq$-linear subspace of $N_\qq$.
Otherwise, we say that $\Sigma$ is {\em non-degenerate}.
The fan in $N_\qq$ is degenerate if and only if $X(\Sigma)$ splits as $X(\Sigma')\times \mathbb{T}$
for some torus $\mathbb{T}$.
If $X(\Sigma)$ is defined by a non-degenerate fan $\Sigma$, then
$X(\Sigma)$ has only constant invertible functions.}
\end{definition} 

The following proposition is well-known, see e.g.,~\cite{HK00}.

\begin{proposition}
A Mori dream space $X$ is a toric variety
if and only if its Cox rings is a polynomial ring.
\end{proposition}

\begin{definition}{\em 
The spectrum of the Cox ring is called the {\em total coordinate space},
and is usually denoted by $\bar{X}$.
Note that $\bar{X}$ admits an action of the {\em characteristic quasi-torus}
$\mathbb{T}_{{\rm Cl}(X)}:={\rm Spec}(\kk[{\rm Cl}(X)])$.
We can reconstruct $X$ as the quotient of a big open subset
$\hat{X}\subset \bar{X}$ by the characteristic quasi-torus action $\mathbb{T}_{{\rm Cl}(X)}$.
In the case that $X$ is affine, then we have that $\hat{X}=\bar{X}$ (see, e.g.,~\cite{ADHL15}*{1.6.3.4}.
Thus, any $n$-dimensional affine toric variety is isomorphic to the quotient
of $\mathbb{A}^{\rho+n}$ by the characteristic quasi-torus $\mathbb{T}_{{\rm Cl}(X)}$.
Here, $\rho$ is the Picard rank of the affine toric variety.
}
\end{definition} 

\begin{remark}{\em 
From now on, when we say that $x\in X$ is an affine toric variety, we mean that $X$ is an affine toric variety defined by a non-degenerate pointed polyhedral cone $\sigma \subset N_\qq$ and 
$x$ is the unique closed fixed point.
}
\end{remark}

\subsection{Automorphisms of toric varieties}
In this subsection, we will prove some statements about
automorphisms of toric varieties.
Some of the lemmas proved in this subsection are well-known, we give a short argument for the sake of completeness (see, e.g.,~\cites{AB13,ADHL15,SMS19}).
In many cases, we prove them in a slightly more general setting.

\begin{definition}{\em 
Let $X$ be an affine variety and $x\in X$ a closed point.
We denote its {\em automorphism group} by ${\rm Aut}(X,x)$,
i.e., the subgroup of ${\rm Aut}(X)$ which fixes $x$.
If $x\in X$ is clear form the context, we may omit it in the notation and just write $\widehat{X}$.
The group ${\rm Aut}(\widehat{X},x)$ is the group of automorphisms of the germ
that fix $x$.
}
\end{definition}

\begin{remark}{\em 
In general,
the groups ${\rm Aut}(X,x)$ and ${\rm Aut}(\widehat{X},x)$
are not algebraic groups. 
This phenomenon makes it hard to work with finite subgroups of both of them.
In the case of ${\rm Aut}(\widehat{X},x)$,
we know that it contains a {\em Levi subgroup}, i.e.,
a maximal reductive automorphism group (see. e.g.,~\cite{HM89}).
Any algebraic reductive subgroup of ${\rm Aut}(\widehat{X},x)$ is contained
in the Levi subgroup up to conjugation.
In particular, for an affine toric variety $x\in X$,
the image of $\mathbb{T}_X$ in ${\rm Aut}(\widehat{X},x)$ is contained
in the Levi subgroup up to conjugation.
In this article, instead of working directly with this Levi subgroup,
we will find a special linear algebraic subgroup of ${\rm Aut}(X,x)$
containing our finite group.
}
\end{remark}

\begin{definition}\label{def:form-equiv}{\em 
Let $G$ be an algebraic group acting on a variety $X$
and $x\in X$ be a fixed point.
Let $H$ be an algebraic group acting on a variety $Y$
and $y\in Y$ be a fixed point.
We say that the action of $G$ on $X$ around $x$
is {\em formally equivalent} to the action
of $H$ on $Y$ around $y$, 
if the following condition is satisfied:
There is a group isomorphism $\phi\colon G\rightarrow H$
and an isomorphism $f\colon \widehat{X}_x\rightarrow \widehat{Y}_y$ which is equivariant,
i.e., the relation $fg=\phi(g)f$ holds.
}
\end{definition} 

\begin{lemma}\label{lem:injectivity} 
Let $X$ be a smooth toric variety and
let $Z$ be a smooth variety.
Let $\pi_1$ and $\pi_2$ be the canonical projections of $X\times Z$.
Then, the homomorphism
${\rm Cl}(X)\times {\rm Cl}(Z)\rightarrow 
{\rm Cl}(X\times Z)$ given by
$(L_1,L_2)\mapsto \pi_1^*L_1\otimes \pi_2^*L_2$ is an isomorphism.
\end{lemma} 

\begin{proof}
Let $H_1,\dots,H_k$ be the torus invariant divisors of $X$.
For each $i$, we let $\bar{H}_i$
be the pull-back of $H_i$ to $X\times Z$.
The injectivity of the above homomorphism follows by intersecting with closed fibers of $\pi_1$ and $\pi_2$.
Let $D$ be a prime divisor on $X\times Z$.
If $D$ is vertical over $Z$, then it is the pull-back of a divisor on $Z$.
On the other hand, if $D$ dominates $Z$, 
then we can write $D\sim_Z \sum_{i=1}^kn_i\bar{H}_i$.
Thus, the above homomorphism is surjective.
\end{proof}

\begin{lemma}\label{lem:action-aut-0}
Let $x\in X$ be an affine toric variety.
Let $G\leqslant {\rm Aut}(X,x)$ be a connected linear subgroup.
Let $\tau \in G$.
Then, $\tau$ acts trivially on ${\rm Cl}(X,x)$.
\end{lemma}

\begin{proof}
Let $H_1,\dots, H_k$ be the torus invariant divisors of $X$.
We set $Z:=G^0(X,x)$.
Note that $Z$ is a smooth algebraic variety.
Let $\phi\colon X\times Z \rightarrow X\times Z$ be the universal automorphism, i.e.,
the morphism given by
$(x,\tau)\mapsto (\tau(x),\tau)$.
Note that every automorphism of $X$ fixes the singular locus.
Then, the universal automorphism induces
an automorphism $\phi\colon X^{\rm sm}\times Z\rightarrow X^{\rm sm}\times Z$.
By Lemma~\ref{lem:injectivity}, we have that
\[
\phi^*\pi_2^*\mathcal{O}_{X^{\rm sm}}(H_i)
\simeq \pi_1^* \mathcal{L}_1 \otimes \pi_2^* \mathcal{L}_2,
\]
for certain line bundles
$\mathcal{L}_1$ on $X^{\rm sm}$
and 
$\mathcal{L}_2$ on $Z$.
Restricting the above isomorphism to $X^{\rm sm}\times \{{\rm id}\}$
and $X^{\rm sm}\times \{ \tau\}$, 
we conclude that $\mathcal{O}_{X^{\rm sm}}(H_i)\simeq \mathcal{L}_1$
and $\tau^*\mathcal{O}_{X^{\rm sm}}(H_i)\simeq \mathcal{L}_1 \simeq \mathcal{O}_{X^{\rm sm}}(H_i)$.
Thus, we have that
$\tau^*\mathcal{O}_{X^{\rm sm}}(H_i)\simeq \mathcal{O}_{X^{\rm sm}}(H_i)$.
Since the $\mathcal{O}_{X^{\rm sm}}(H_i)$ generate ${\rm Cl}(X,x)$,
we conclude that $\tau$ acts as the identity on ${\rm Cl}(X,x)$.
\end{proof}

\begin{definition}
{\em 
Let $x\in X$ be an affine toric variety
and $x$ its unique fixed point.
Write $X={\rm Spec}(R)$.
Let $f\in {\rm Aut}(X,x)$ be an automorphism.
We say that $\tilde{f}\in {\rm Aut}(\mathbb{A}^{\rho+n},0)$
is a {\em lifting} of $f$ to the Cox ring,
if $\tilde{f}$ normalizes the characteristic quasi-torus
and the automorphism that $\tilde{f}$ induces on 
\[
R=\kk[x_1,\dots,x_{\rho+n}]^{{\rm Cl}(X)}
\]
equals $f$.
}
\end{definition}

\begin{lemma}\label{lem:action-characteristic}
Let $x\in X$ be an affine toric variety and
let $\mathbb{A}^{\rho+n}$ be its Cox ring.
Let $f\in {\rm Aut}(X,x)$ and $\tilde{f}\in {\rm Aut}(\mathbb{A}^{\rho+n},0)$ a lifting.
Then, the action of $\tilde{f}$ on $\mathbb{T}_{{\rm Cl}(X)}$ by conjugation
is induced by the action of $f^*$ on ${\rm Cl}(X)$.
In particular, if $f^*$ acts trivially on ${\rm Cl}(X)$,
then $\tilde{f}$ commutes with the characteristic quasi-torus.
\end{lemma}

\begin{proof}
Let $h\in \kk[x_1,\dots,x_{\rho+n}]\simeq {\rm Cox}(X)$ be a homogeneous element
of degree $w=[D]\in {\rm Cl}(X)$.
An element $\chi^w$ of the characteristic quasi-torus $\mathbb{T}_{{\rm Cl}(X)}$ acts on $h$ by
$t\cdot h = ht =\chi^w(t)h.$
On the other hand, we have that
\[
(\tilde{f}t\tilde{f}^{-1})\cdot h =
t\cdot (h\tilde{f}) \tilde{f}^{-1} =
\chi^{{\rm deg}(hg)}(t)h =\chi^{g^*(w)}(t)h.
\]
This equality proves the lemma.
\end{proof}

Now, we turn to prove some propositions about lifting
automorphisms of an affine toric variety to its Cox ring.

\begin{proposition}\label{prop:lifting-aut}
Let $x\in X$ be an affine toric variety.
Let $G\leqslant {\rm Aut}(X,x)$ 
be a reductive subgroup. 
Then, there is an exact sequence
\[
1\rightarrow
\mathbb{T}_{{\rm Cl}(X)}
\rightarrow 
\widetilde{G}
\rightarrow 
G
\rightarrow 
1.
\]
where $\widetilde{G}$ is a reductive subgroup of ${\rm Aut}(\mathbb{A}^{\rho+n},0)$
\end{proposition}

\begin{proof}
First, we proceed to lift the connected component $G_0$ of $G$.
Note that $G_0$ is a connected linear algebraic reductive group.

Let $G_0'\rightarrow G_0$ be the lifting of 
$G_0$ to ${\rm Aut}(\mathbb{A}^{\rho+n})$ constructed in~\cite{ADHL15}*{4.2.3.1.}.
We have an exact sequence
\[
1\rightarrow F \rightarrow G_0'
\rightarrow G_0\rightarrow 1,
\]
where $F$ is a finite group.
Furthermore, $F$ is the intersection of $G'_0$
with the Picard quasi-torus (see, e.g.,~\cite{AG10}*{Theorem 5.1}).
Since $G_0$ is connected, 
by Lemma~\ref{lem:action-aut-0}, 
we conclude that it acts trivially on the characteristic quasi-torus.
By Lemma~\ref{lem:action-characteristic}, we conclude
that $G'_0$ commutes with the characteristic quasi-torus
$\mathbb{T}_{{\rm Cl}(X)}$.

We claim that $G'_0$ is actually a subgroup of
${\rm Aut}(\mathbb{A}^{\rho+n},0)$.
Let $g'\in G'_0$ be the lifting of $g\in G_0$. 
Since $g'$ commutes with the characteristic torus, 
then $g'(0)$ is a characteristic quasi-torus invariant point.
If $0\in \mathbb{A}^{\rho+n}$ is the only $\mathbb{T}_{{\rm Cl}(X)}$ invariant point, then we are done.
Otherwise, we can write $\mathbb{A}^{\rho+n}\simeq \mathbb{A}^{\rho+n-k}\times \mathbb{A}^k$, where $k\geq 1$
and $\mathbb{T}_{{\rm Cl}(X)}$ acts trivially on $\mathbb{A}^k$.
In this case $X\simeq Y\times \mathbb{A}^k$.
Since $g$ fix the origin of $\mathbb{A}^k$,
then the same holds for the $g'$.
We conclude that $G'_0 \leqslant {\rm Aut}(\mathbb{A}^{\rho+n},0)$.

Let $\widetilde{G}_0$ the subgroup of 
${\rm Aut}(\mathbb{A}^{\rho+n},0)$ generated
by $\mathbb{T}_{{\rm Cl}(X)}$ and $G'_0$.
We claim that $\widetilde{G}_0$ is a linear algebraic group.
Indeed, since $G'_0$ and $\mathbb{T}_{{\rm Cl}(X)}$ 
commute in ${\rm Aut}(\mathbb{A}^{\rho+n},0)$,
we conclude that there is a homomorphism
\[
\mathbb{T}_{{\rm Cl}(X)}\times G'_0\rightarrow
{\rm Aut}(\mathbb{A}^{\rho+n},0)
\]
given by $(t,g)\mapsto tg$.
The image of this homomorphism equals $\widetilde{G}_0$.
Furthermore, the kernel of such epimorphism is
$F\leqslant \mathbb{T}_{{\rm Cl}(X)}\times G'_0$.
Thus $\widetilde{G}_0$ is the quotient of a connected linear algebraic group
by a finite group.
We conclude that $\widetilde{G}_0$ is a connected linear algebraic group.

We claim that $\widetilde{G}_0$ is reductive.
Let $U\leqslant \widetilde{G}_0$
be a smooth connected unipotent normal subgroup.
Since the torus contains no unipotent element, we conclude that $U\cap \mathbb{T}_{{\rm Cl}(X)}=\{e\}$.
Thus, the image of $U$ in $G_0$ is a smooth connected unipotent normal subgroup.
Since $G_0$ is reductive, we conclude that the image of $U$ in $G_0$ is trivial, 
hence $U$ is trivial.

Now, we proceed to lift $G$ to $\widetilde{G}\leqslant {\rm Aut}(\mathbb{A}^{\rho+n},0)$.
By~\cite{AG10}*{Theorem 5.1}, we have an exact sequence
\[
1\rightarrow \mathbb{T}_{{\rm Cl}(X)}
\rightarrow \widetilde{{\rm Aut}}(\mathbb{A}^{\rho+n})
\rightarrow {\rm Aut}(X)\rightarrow 1, 
\]
where $\widetilde{{\rm Aut}}(\mathbb{A}^{\rho+n})$
is the subgroup of ${\rm Aut}(\mathbb{A}^{\rho+n})$
which normalizes the characteristic quasi-torus.
We consider the projection homomorphism
 $\pi\colon  {\rm Aut}(\mathbb{A}^{\rho+n})\rightarrow {\rm Aut}(X)$.
Note that $\pi^{-1}(G_0)=\widetilde{G}_0$.
We let $\widetilde{G}$ to be the pre-image of
$G\leqslant {\rm Aut}(X,x)$ on
$\widetilde{{\rm Aut}}(\mathbb{A}^{\rho+n})$.
Note that $\widetilde{G}$ normalizes
the characteristic quasi-torus.
Hence, we conclude that
$\widetilde{G}\leqslant {\rm Aut}(\mathbb{A}^{\rho+n},0)$.
On the other hand, we have an exact sequence
\[
1\rightarrow
\widetilde{G}_0\rightarrow \widetilde{G}
\rightarrow F'\rightarrow 1,
\]
where $F'$ is a finite group.
We conclude that the induced representation 
on $\widetilde{G}$ is faithful.
Hence $\widetilde{G}$ is a linear algebraic group.
Since $\widetilde{G}_0$ is reductive,
we conclude that $\widetilde{G}$ is reductive as well.
\end{proof}

It is known that the automorphism group of an affine toric variety is not an algebraic group.
Indeed, it may not be of finite type over $\kk$.
Now, we turn to prove that inside the automorphism group of an affine toric variety $x\in X$, 
there is an special linear algebraic group
$G(X,x)$.
The action of any reductive group $G$
can be realized as the action of a subgroup of
$G(X,x)$ up to formal equivalence of group actions.

\begin{definition}\label{def:g(X,x)}
{\em
Let $x\in X$ be an affine toric variety.
Let $\mathbb{A}^{\rho+n}$ its total coordinate space 
and $\mathbb{T}_{{\rm Cl}(X)}$ the characteristic quasi-torus acting on the
total coordinate space.
We have a natural embedding 
$\mathbb{T}_{{\rm Cl}(X)}\leqslant {\rm GL}_{\rho+n}(\kk)$.
Let $N=N_{\mathbb{T}_{{\rm Cl}(X)}}({\rm GL}_{\rho+n}(\kk))$
be the normalizer of the characteristic quasi-torus.
We define 
$G(X,x)=N/\mathbb{T}_{{\rm Cl}(X)}$.
Note that $G(X,x)$ is a subgroup of ${\rm Aut}(X,x)$.
On the other hand, 
$G(X,x)$ is a linear algebraic group by construction.
We say that $G(X,x)$ is the group of {\em linear automorphisms} of the affine toric variety.
Indeed, $G(X,x)$ is exactly the group of automorphisms which lift to linear automorphisms on the Cox ring.
We will denote by $G^0(X,x)$ its connected component.

Denote by $\mathbb{T}_{{\rm Cox}(X)}$ a maximal torus of $\mathbb{A}_{\rho+n}$
which we may assume is embedded in ${\rm GL}_{\rho+n}(\kk)$.
Note that 
$\mathbb{T}_{{\rm Cox}(X)}\leqslant N$.
Furthermore, the image of 
$\mathbb{T}_{{\rm Cox}(X)}$ on
${\rm Aut}(X,x)$ coincides with
$\mathbb{T}_X$.
Hence, we have that 
$\mathbb{T}_X\leqslant G(X,x)$.
In particular, we have that
$\mathbb{T}_X\leqslant G^0(X,x)$.
Note that, if $x\in X$ is the affine space,
then $G(X,x)$ is simply ${\rm GL}_n(\kk)$.
}
\end{definition} 

\begin{lemma}\label{lem:diag}
Let $x\in X$ be an affine toric variety.
Let $G\leqslant {\rm Aut}(X,x)$ be a reductive subgroup.
Up to formal equivalence of group actions, 
we may assume that 
\[
G\leqslant G(X,x)\leqslant {\rm Aut}(X,x).
\]
\end{lemma}

\begin{proof}
By Proposition~\ref{prop:lifting-aut}, we may lift the action
of $G$ on $(X,x)$ to $\widetilde{G}$ on $(\mathbb{A}^{\rho+n},0)$.
The group $\widetilde{G}$ is reductive,
hence by Luna's slice theorem, 
the action of $\widetilde{G}$ on the local completion of $\mathbb{A}^{\rho+n}$
is equivalent to the action of a reductive subgroup of ${\rm GL}_{\rho+n}(\kk)$.
Thus, we have a commutative diagram:
\begin{equation}\label{diag} 
\xymatrix{
\widetilde{G} \acts (\mathbb{A}^{\rho+n},0) \ar[r]\ar[d]_-{/\mathbb{T}_{{\rm Cl}(X)}} & 
\widetilde{G} \acts \widehat{\mathbb{A}}^{\rho+n}_0\ar[r]^-{\cong} \ar[d]_-{/\mathbb{T}_{{\rm Cl}(X)}} &\widehat{\mathbb{A}}^{\rho+n}_0 \actsr \widetilde{G}' \ar[d]^-{/\mathbb{T}_{{\rm Cl}(X)}}&
(\mathbb{A}^{\rho+n},0) \actsr \widetilde{G}' \ar[l]\ar[d]^-{/\mathbb{T}_{{\rm Cl}(X)}} \\
G\acts (X,x) \ar[r] &
G\acts \widehat{X}_x \ar[r]^-{\cong} & 
\widehat{X}_x \actsr  G' & 
(X,x)\actsr G' \ar[l].
}
\end{equation}
Here, $\widetilde{G}'\leqslant {\rm GL}_{\rho+n}(\kk)$ is a subgroup
normalizing the characteristic quasi-torus $\mathbb{T}_{{\rm Cl}(X)}\leqslant 
{\rm GL}_{\rho+n}(\kk)$.
The middle isomorphisms are equivariant with respect to the reductive group actions.
Furthermore, $\widetilde{G}\simeq \widetilde{G}'$
and $G\simeq G'$.
We conclude that $G'\leqslant G(X,x)\leqslant {\rm Aut}(X,x)$.
This proves the lemma.
\end{proof}

\begin{remark}{\em 
Let $x\in X$ be an affine toric variety
and $G\leqslant {\rm Aut}(X,x)$ be a finite subgroup.
Lemma~\ref{lem:diag} says that, if we want to study the action
of $G$ formally around the point $x\in X$, then
we can always assume the existence of a linear algebraic group
$G\leqslant G(X,x)\leqslant {\rm Aut}(X,x)$ so that
$G^0(X,x)\geqslant \mathbb{T}_X$.
This Lemma is a natural generalization of the fact that 
in order to study quotient singularities, we
can assume that the acting group holds $G\leqslant {\rm GL}_n(\kk)$
(see, e.g.,~\cite{Kol13}*{3.2}).
Another advantage of Lemma~\ref{lem:diag} is that it allows us
to keep working with affine automorphisms, instead of 
considering germ automorphisms.
Furthermore, $G(X,x)$ will naturally contain some automorphisms of $(X,x)$.
Note that the construction of $G(X,x)$ only depends
on the affine variety $x\in X$.
The next step is to understand the 
discrete group $G(X,x)/G^0(X,x)$
using the geometry of the affine toric variety $X$.
}
\end{remark}

\begin{definition}{\em 
Let $X$ and $X'$ be two toric varieties
with character spaces $M$ and $M'$ respectively.
In this case, we consider
$\mathbb{T}_X=T_M$ and 
$\mathbb{T}_{X'}=T_{M'}$.
An {\em outer toric morphism} is a pair $(f,\tau)$ 
where $f\colon X\rightarrow X'$ is a morphism
and $\tau\colon \mathbb{T}_X\rightarrow \mathbb{T}_{X'}$ is a group morphism
so that $f(\lambda \cdot x)=\tau(\lambda)\cdot f(x)$
for all $x\in X$ and $\lambda\in \mathbb{T}_X$.
We say that an outer toric morphism
is an {\em outer toric automorphism} if $f$ and $\tau$
are isomorphisms.
We say that an outer toric automorphism 
is a {\em toric automorphism} if $\tau$ is the identity.
We denote by ${\rm Aut}_{\rm out}(X)$ the group of
outer toric automorphisms of $X$.
In particular, we have that
$\mathbb{T}_X \leqslant {\rm Aut}_{\rm out}(X)$,
where $n$ is the dimension of $X$.
}
\end{definition}

\begin{definition}{\em 
Let $X$ be a toric variety and
$\Sigma\subset N_\qq$ be its defining fan.
Let $\tau \colon M\rightarrow M$ be an automorphism
of the characters lattice $M$.
Then, we have an induced transpose automorphism
$\tau^* \colon N_\qq \rightarrow N_\qq$.
We denote by ${\rm Aut}_\Sigma(M)$ the subgroup of
automorphisms of $M$ for which
$\tau^*\sigma \in \Sigma$ for every $\sigma \in \Sigma$.
Then, every element of ${\rm Aut}_\Sigma(M)$ induces
an element of ${\rm Aut}_{\rm out}(M)$.
Note that we have a natural monomorphism
${\rm Aut}_\Sigma(M)\rightarrow {\rm Aut}_{\rm out}(X)$.
In Lemma~\ref{lem:aut-sigma-out}, we will see that
${\rm Aut}_{\rm out}(X)$ is essentially the subgroup of
${\rm Aut}(\mathbb{T}_X)$ which extends to ${\rm Aut}(X,x)$
and ${\rm Aut}_{\Sigma}(M)\cap \mathbb{T}_X=\{e\}$.
}
\end{definition}

\begin{lemma}\label{lem:aut-sigma-out}
Let $X$ be a toric variety of dimension $n$
with defining fan $\Sigma$.
We have an exact sequence
\[
1\rightarrow \mathbb{T}_X \rightarrow {\rm Aut}_{\rm out}(X)
\rightarrow {\rm Aut}_\Sigma(M)\rightarrow 1.
\]
\end{lemma}

\begin{proof}
It is clear that $\mathbb{T}_X$ is a normal subgroup of ${\rm Aut}_{\rm out}(X)$. 
Note that every element of ${\rm Aut}_{\rm out}(X)$ induces an automorphism of $\mathbb{T}_X\leqslant X$.
The image of the restriction homomorphism
${\rm Aut}_{\rm out}(X)\rightarrow {\rm Aut}(\mathbb{T}_X)$
equals ${\rm Aut}_{\Sigma}(M)$.
Two elements of ${\rm Aut}_{\rm out}(X)$ restrict to the same automorphism of the torus if and only if they differ
by an element of $\mathbb{T}_X$.
\end{proof}

\begin{remark}
{\em 
Let $X$ be an affine toric variety
and $x\in X$ its unique closed point.
The group ${\rm Aut}_{\rm out}(X)$ is contained in $G(X,x)$,
where $G(X,x)$ is as in 
Definition~\ref{def:g(X,x)}.
Indeed, both $\mathbb{T}_X$ and ${\rm Aut}_{\Sigma}(M)$
are contained in $G(X,x)$.
The former, since $\mathbb{T}_{{\rm Cox}(X)}\leqslant N$.
The latter, since ${\rm Aut}_{\Sigma}(M)$ lifts to the Cox ring 
as a subgroup of $S_{\rho+n}$ acting on the set of variables $x_1,\dots,x_{\rho+n}$
that normalizes the characteristic quasi-torus.

Moreover, the group ${\rm Aut}_{\rm out}(X)$ is reductive.
Although, it may not be connected.
In any case, we have a sequence of subgroups
\[
\mathbb{T}_X \leqslant {\rm Aut}_{\rm out}(X)
\leqslant G(X,x)\leqslant {\rm Aut}(X,x).
\]}
\end{remark}

\begin{lemma}\label{comp-g-g0}
Let $G$ be a linear algebraic group
and $\mathbb{T}\leqslant G$ be a maximal torus.
For any $f\in G$, we can find $f_0\in G^0$ 
and $\tau \in {\rm Aut}(\mathbb{T})$ so that
the automorphism $\bar{f}=f_0 f$, holds that
\begin{equation}
\bar{f}t = \tau(t)\bar{f}
\end{equation}
for every $t\in \mathbb{T}$.
\end{lemma}

\begin{proof}
Let $f$ be an element of $G$.
$f$ induces an automorphism 
$\phi$ on $G^0$ by the conjugation
$g\mapsto fgf^{-1}$.
The torus
$\phi(\mathbb{T})$ is a maximal torus of $G^0$, hence $f_0\phi(\mathbb{T})f_0^{-1}=\mathbb{T}$ for some element $f_0\in G^0$.
Define $\bar{f}=f_0f$.
Note that 
\[
\mathbb{T}\rightarrow
f\mathbb{T}f^{-1}\rightarrow
f_0f\mathbb{T}f^{-1}f_0^{-1}=\mathbb{T}
\]
is an automorphism of the torus
which we denote by $\tau$.
Let $t\in \mathbb{T}$ be an element of the torus.
Then, we have that
\[
\tau(t)\bar{f}=
f_0ftf^{-1}f_0^{-1}f_0f=f_0ft=\bar{f}t.
\]
\end{proof}

\begin{proposition}\label{g-over-g0}
Let $x\in X$ be an affine toric variety.
Then, there is a surjective homomorphism
\[
{\rm Aut}_{\Sigma}(M)\rightarrow
G(X,x)/G^0(X,x).
\]
\end{proposition}

\begin{proof}
Note that we have a homomorphisms
\[
{\rm Aut}_{\rm out}(X)/\mathbb{T}_X
\rightarrow 
{\rm Aut}_{\rm out}(X)/({\rm Aut}_{\rm out}(X)\cap G^0(X,x)) 
\rightarrow 
G(X,x)/G^0(X,x).
\]
We claim that this composition is surjective.
Let $f\in G(X,x)$.
Let $f_0\in G(X,x)$ be the element constructed in Lemma~\ref{comp-g-g0}
and $\bar{f}=f_0f$.
Then, there exists an automorphism $\tau$ of $\mathbb{T}_X$ so that
\[
\bar{f}t=\tau(t)\bar{f}
\]
for every element $t\in \mathbb{T}_X$.
Hence, $\bar{f}\in {\rm Aut}_{\rm out}(X)$.
We conclude that the image of $\bar{f}$ and $f$ on $G(X,x)/G^0(X,x)$ are the same.
Hence, the above composition is surjective.
By Lemma~\ref{lem:aut-sigma-out}, we conclude that there is a surjective homomorphism
\[
{\rm Aut}_{\rm \Sigma}(M)\rightarrow G(X,x)/G^0(X,x).
\]
\end{proof}

The above Proposition allows us to prove the following Lemma about the discrete group
$G(X,x)/G^0(X,x)$.

\begin{lemma}\label{lem:control-G/G^0}
Let $x\in X$ be a $n$-dimensional affine toric variety.
Then, there exists a constant $k(n)$,
only depending on $n$,
satisfying the following.
Every element of the group
$G(X,x)/G^0(X,x)$ has order at most $k(n)$.
\end{lemma}

\begin{proof}
We know that $G(X,x)/G^0(X,x)$ is a finite group,
so any element has finite order.
By Proposition~\ref{g-over-g0}, we know that
${\rm Aut}_{\Sigma}(M)\rightarrow G(X,x)/G^0(X,x)$
is a epimorphism.
Let $g\in G(X,x)/G^0(X,x)$ and 
$\tilde{g}\in {\rm Aut}_\Sigma(M)$ an
element mapping to $g$.
Note that $\tilde{g}$ has finite order.
Indeed, we have a monomorphism
${\rm Aut}_{\Sigma}(M)\rightarrow S_{|\Sigma(1)|}$
given by associating to an element of ${\rm Aut}_{\Sigma}(M)$
the corresponding permutation of $\Sigma(1)$.
Then, $\tilde{g}$ is a finite order element of
${\rm Aut}_{\Sigma}(M)\leqslant {\rm Aut}(M)\simeq{\rm GL}_n(\zz)$.
In particular, the order of $\tilde{g}$ is bounded by a function $k(n)$ on $n$.
We conclude that the order of $g$
is bounded by a function $k(n)$ on $n$.
This proves the Lemma.
\end{proof}

We finish this subsection on automorphisms of affine toric varieties, by studying 
finite automorphisms which lifts as the identity on the exceptional divisor of a toric blow-up.

\begin{proposition}
\label{prop:normal-subgroup-identity}
Let $x\in X$ be a $n$-dimensional 
affine toric variety and
$Y\rightarrow X$ be a projective birational toric morphism extracts a divisor $E$ over $x$.
Let $G\leqslant {\rm Aut}(X,x)$ be a finite subgroup.
Assume that $Y\rightarrow X$ is $G$-equivariant
and $G$ fixes point-wise an exceptional divisor $E\subset Y$.
Then, up to conjugation, we have that
\[
G\leqslant \mathbb{T}_0
\leqslant \mathbb{T}_X 
\leqslant {\rm Aut}(\widehat{X},x),
\]
where $\mathbb{T}_0\leqslant \mathbb{T}_X$ is a one-dimensional sub-torus.
\end{proposition}

\begin{proof}
We claim that $G$ is cyclic group.
Indeed, let $e\in E$ be a smooth point.
Then, $G$ has a faithful representation in $T_eE$ fixing a hyperplane pointwise
(see, e.g.,~\cite{FZ05}*{Lemma 2.6.(b)}).
We conclude that $G$ is either cyclic.

By Lemma~\ref{lem:diag}, 
we may assume that $G$ is a subgroup of the linear algebraic group $G(X,x)$.
We denote by $G'(X,x)$ the subgroup of $G(X,x)$ corresponding to elements which lift to $Y$, i.e., elements which preserve the ideal defining the blow-up.
Then, $G'(X,x)$ is a linear algebraic group.
Let $G''(X,x)\leqslant G'(X,x)$ be the subgroup generated by elements which 
restrict to the identity on $E$.
Since $E$ is a projective toric variety, we can find a finite number of points $e_1,\dots,e_k$ so that an automorphism of $E$ is the identity if and only if it fixes such points.
Then, $G''(X,x)$ equals the subgroup of $G'(X,x)$ which fixes the points $e_1,\dots,e_k\in Y$.
In particular, $G''(X,x)$ is a linear algebraic group.
Note that $G\leqslant G''(X,x)$.

We claim that $G''(X,x)$ contains a one-dimensional torus.
Since $Y\rightarrow X$ is a toric morphism, then we have that 
$\mathbb{T}_X$ acts on $Y$
and fixes the divisor $E$.
The torus $\mathbb{T}_E$ of $E$ has
dimension $n-1$.
Hence, the surjective homomorphism
$\mathbb{T}_X\rightarrow \mathbb{T}_E$,
induced by the restriction to $E$,
has kernel a one-dimensional torus $\mathbb{T}_0$.
Note that $\mathbb{T}_0$ acts trivially on $E$.
Hence, $\mathbb{T}_0$ is contained in the connected component of $G''(X,x)$.
By dimension reasons, a maximal torus of the connected component of $G''(X,x)$ is at most one dimensional.
Indeed, the dimension of the acting torus is at most the codimension 
of the fixed point set.
Thus, $\mathbb{T}_0$ is a maximal torus of the connected component of $G''(X,x)$. 

We claim that $G''(X,x)$ is connected.
We denote by $G''_0(X,x)$ its connected component.
Let $f\in G''(X,x)$ be any element.
By Lemma~\ref{comp-g-g0}, we know that there exists $f_0\in G''_0(X,x)$ so that
$\bar{f}=f_0f$ satisfies
\[
\bar{f}t=\tau(t)\bar{f}
\]
for an automorphism $\tau$ of $\mathbb{T}_0$. 
We show that for every point
$y\in Y$, we have that
\begin{equation}\label{int-orb} 
\overline{\mathbb{T}_0\cdot \bar{f}(y)}
\cap 
\overline{\mathbb{T}_0\cdot y}\neq \emptyset.
\end{equation}
Let $y\in Y$ be a point outside $E$.
Let 
\[
y_0 =\overline{\mathbb{T}_0\cdot y}\cap E
\text{ and } y_1=\overline{\mathbb{T}_0\cdot \bar{f}(y)}\cap E.
\]
Then, we have that 
\[
y_0=
\bar{f}(y_0)=
\lim_{t\rightarrow 0}\bar{f}(t\cdot y)=
\lim_{t\rightarrow 0}\tau(t)\cdot \bar{f}(y)=
y_1.
\]
This shows the non-emptyness of equation~\eqref{int-orb} and that $\tau(t)=t$. 
In particular, if there is a $\mathbb{T}_0$-orbit closure which is disjoint from all other 
$\mathbb{T}_0$-orbit closures,
then we have that $\bar{f}$ fixes such $\mathbb{T}_0$-orbit closure.

Let $U_Y\subset Y$ be the complement of all the torus invariant divisors of $Y$ except $E$.
Then $U_Y\simeq \mathbb{G}_m^{n-1}\times \mathbb{A}^1$ and $\mathbb{T}_0$ acts 
on $U_Y$ by multiplication on the last coordinate.
In particular, all orbit closures of $\mathbb{T}_0$ on $U_Y$ are disjoint.
By equation~\eqref{int-orb}, we conclude that on $U_Y$, the automorphism $\bar{f}$
fixes the $\mathbb{T}_0$-orbits.
Hence, $\bar{f}$ is an automorphism of the torus which extends to $Y$.
Thus, $\bar{f}$ is an element of ${\rm Aut}_{\rm out}(X)$.
We conclude that $\bar{f}$ is an automorphism of $\mathbb{G}^n_m$,
that extends to $\mathbb{G}_m^{n-1}\times \mathbb{A}^1$ and acts trivially on $\mathbb{G}_m^{n-1}\times \{0\}$.
In coordinates $(x_1,\dots, x_n)$ of $\mathbb{G}_m^{n-1}\times \mathbb{A}^1$,
we have that \
\[
\bar{f}(x_1,\dots,x_n)=(x_1,\dots,x_{n-1}, \lambda x^\alpha x_n).
\]
Here,
$x^\alpha$ is a monomial on $x_1,\dots,x_{n-1}$ and $\lambda$ is a constant.
By Lemma~\ref{lem:aut-sigma-out}, we conclude that a finite power of $\bar{f}$
belongs to $\mathbb{T}_X$.
Thus, we have that $x^\alpha=1$ and
$\bar{f}\in \mathbb{T}_X$.
Furthermore, the only elements of the torus which fix $E$ pointwise
are those in $\mathbb{T}_0$.
Hence, we conclude that $\bar{f}\in \mathbb{T}_0 \leqslant G''_0(X,x)$.
Thus, for any element $f\in G''(X,x)$, we can find $\bar{f}\in G''_0(X,x)$ and $f_0\in G''_0(X,x)$, so that
$f=f_0^{-1}\bar{f}$.
This proves that $G''_0(X,x)=G''(X,x)$.
Hence, $G''(X,x_0)$ is a connected linear algebraic group.

Since $G$ is a finite cyclic group, 
then it is generated by a semisimple element, so it conjugates inside the maximal torus $\mathbb{T}_0$.
\end{proof}

\subsection{Kawamata log terminal singularities}

In this subsection, we recall the definition of Kawamata log terminal singularities
and recall that toric singularities are klt.

\begin{definition}
{\em A {\em log pair} consist of
$(X,B)$ where $X$ is a normal variety 
and $B$ is an effective $\qq$-divisor 
on $X$ so that $K_X+B$ is a 
$\qq$-Cartier $\qq$-divisor.
}
\end{definition} 

\begin{definition}{\em 
Let $X$ be a normal variety.
A {\em divisor over $X$} is a prime divisor $E\subset Y$,
where $Y\rightarrow X$ is a projective birational morphism
from a normal variety.

Let $(X,B)$ be a log pair and $E$ be a 
divisor over $X$.
We define the {\em log discrepancy}
of $(X,B)$ at $E$ to be
\[
a_E(X,B):=
{\rm coeff}_E(K_Y-\pi^*(K_X+B))
\]
where as usual,
we pick $K_Y$ so that $\pi_*K_Y=K_X$.
The definition of the log discrepancy
does not depend on the model $Y$.

A {\em log resolution} of a log pair 
$(X,B)$ is a projective birational 
morphism $\pi\colon Y\rightarrow X$,
with purely divisorial exceptional locus,
so that
\begin{itemize}
\item $Y$ is a smooth variety, and 
\item the $\pi^{-1}_*\Delta+{\rm Ex}(\pi)_{\rm red}$ has simple normal crossing.
\end{itemize}
By Hironaka's resolution of singularities, we know that any
log pair admits a log resolution.
}
\end{definition}

The log discrepancies of a log pair allow us to measure the singularities of the pair.

\begin{definition}{\em 
A pair $(X,\Delta)$ is said to be {\em Kawamata log terminal} (or {\em klt} for short) if all its log discrepancies are positive, 
i.e., $a_E(X,\Delta)>0$ for every prime divisor $E$ over $X$. It is known that a pair $(X,\Delta)$ is klt if and only if all the log discrepancies corresponding to prime divisors on a log resolution of $(X,\Delta)$ are positive (see, e.g.,~\cite{KM98})
}
\end{definition}

The following result says that toric singularities whose
boundary has coefficients less than one has klt singularities (see, e.g.,~\cite{Amb06}).

\begin{proposition}\label{prop:toric-implies-klt}
Let $(X,B)$ be a toric pair so that $\lfloor B\rfloor=0$.
Then $(X,B)$ has Kawamata log terminal singularities.
\end{proposition}

\subsection{Finite quotients} 

In this subsection, we recall the formula to compute log discrepancies
of divisorial valuations over quotients of klt singularities.

\begin{definition}{\em 
Let $(X,B)$ be a log pair.
We denote by ${\rm Aut}(X,B)$ the subgroup of ${\rm Aut}(X)$
consisting of elements $f\in {\rm Aut}(X)$ for which
$f^*B=B$.
Note that $f$ doesn't need to fix each prime component of $B$,
but at least it must permute prime components with the same coefficient.}
\end{definition}

The following proposition is the Riemann-Hurwitz formula applied to pairs
(see, e.g.,~\cite{Sho92}*{2.1}).

\begin{proposition}\label{prop:pull-back}
Let $f\colon X \rightarrow Y$ be a finite morphism and $(Y,B_Y)$ be a log pair on $Y$.
Write $B_Y=\sum_{i=1}^k d_iD_i$ for its prime decomposition.
Then, we can write
\[
f^*\left(
K_Y+\sum_{i=1}^k d_iD_i 
\right) 
=
K_X +\sum_{i=1}^k \sum_{f(E_j)=D_i} \left( 
1-r_j(1-d_i)
\right) E_j,
\]
where $r_j$ is the ramification index of $f$ at $E_j$.
\end{proposition}

The following proposition follows from Proposition~\ref{prop:pull-back},
see~\cite{Mor20a}*{Proposition 2.18} for the details of the proof.

\begin{proposition}\label{prop:log-quotient}
Let $(X,B)$ be a log pair and $G\leqslant {\rm Aut}(X,B)$ be a finite subgroup.
Let $p\colon X\rightarrow Y:=X/G$ be the quotient morphism.
Then, we can find a boundary $B_Y$ on $Y$ so that
$p^*(K_Y+B_Y)=K_X+B$.
\end{proposition}

\begin{definition}\label{def:log-quotient}{\em 
Let $(X,B)$ be a log pair and $G\leqslant {\rm Aut}(X,B)$ be a finite subgroup.
The pair $(Y,B_Y)$ induced in the quotient in Proposition~\ref{prop:log-quotient} is called the {\em log quotient}
of the pair $(X,B)$ with respect to the action of $G$.}
\end{definition}

\begin{definition}\label{def:toric-quotient}
{\em
Let $x\in X$ be an affine toric variety.
Let $G\leqslant {\rm Aut}(X,x)$ be a finite subgroup.
The log quotient $y\in  (Y,B_Y)$
of $x\in X$ by $G$ is called a
{\rm toric quotient} singularity.
For instance, every quotient singularity
is a toric quotient singularity.
}
\end{definition}

\begin{proposition}\label{prop:quotient-toric}
Let $(X,B)$ be a toric variety and $F\leqslant \mathbb{T}_X$ be a finite subgroup. 
Let $(Y,B_Y)$ be the log quotient of $(X,B)$ by $F$.
Then $(Y,B_Y)$ is a toric pair with respect to the torus action of $\mathbb{T}_Y=\mathbb{T}_X/F$.
Furthermore, if $(X,B)$ is formally toric at $x$,
then $(Y,B_Y)$ is formally toric at $y$.
\end{proposition}

\begin{proof}
We can write $X={\rm Spec}(\kk[x^{m_1},\dots,x^{m_k}])$,
where $x^{m_1},\dots,x^{m_k}$ are monomials in the variables $x_1,\dots,x_n$.
The boundary $B$ is defined by the vanishing set of the variables $x_i$'s.
The finite action of $F$ defines a grading of 
$\kk[x^{m_1},\dots,x^{m_k}]$ into a finite abelian group.
The ring of invariants $\kk[x^{m_1},\dots,x^{m_k}]^F$
is isomorphic to
$\kk[x^{m'_1},\dots,x^{m'_k}]$, where the $x^{m'_i}$'s are monomials on the $x^{m_i}$'s.
In particular, $Y:={\rm Spec}(\kk[x^{m'_1},\dots,x^{m'_k}])$
is a monomial ring in the variables $x_1,\dots,x_n$.
We conclude that $Y$ is a toric variety
with respect to the torus action of $\mathbb{T}_Y={\mathbb{T}_X}/F$.
It is clear that the push-forward of $B$ to $Y$ is defined by the vanishing set of the variables $x_i$'s. Hence, $B_Y$ is a toric boundary on $Y$.

For the second statement, the same argument works after replacing brackets
with double brackets.
\end{proof}

\begin{proposition}\label{prop:ld-finite-quotient}
Let $(X,B)$ be a $G$-invariant pair and $(Y,B_Y)$ its log quotient.
Let $E$ be a $G$-invariant divisorial valuation over $X$ 
and $E_Y$ the induced divisorial valuation on $Y$.
Then, we have that
\[
a_{E_Y}(Y,B_Y) = \frac{a_E(X,B)}{r},
\]
where $r$ is the ramification index of the quotient morphism at $E$.
In particular, $r\leq |G|$.
\end{proposition}

The following proposition is well-known, see e.g.,~\cite{dFD14}*{Proposition 6.2}.

\begin{proposition}\label{prop:non-terminal-toric}
Let $(X,B)$ be a toric pair with $\lfloor B\rfloor=0$.
Let $E$ be a prime divisor over $X$.
Asume that $a_E(X,B)<1$.
Then, $E$ is a toric divisorial valuation over $X$.
In particular, there exists a projective toric birational morphism $Y\rightarrow X$
which only extracts $E$.
\end{proposition}

\begin{proposition}
\label{prop:N-equiv-extraction}
Let $x\in X$ be an affine toric variety.
Let $F\leqslant {\rm Aut}(X,x)$ be a finite subgroup.
Let $E$ be a divisor over $X$ with center $x\in X$.
Assume that $a_E(X,x)=a<1$.
Then, there exists a projective birational 
$F$-equivariant toric morphism 
$Y\rightarrow X$, so that
\begin{enumerate} 
\item $E$ is extracted on $Y$, and 
\item the $F$-orbit of $E$ on $Y$ is the exceptional locus of $Y\rightarrow X$.
\end{enumerate}
In particular, $a_{E'}(X,x)=a$ for each prime exceptional divisor $E'$ on $Y$ over $X$.
\end{proposition}

\begin{proof}
Let $x' \in (X',\Delta')$ be the quotient of $x\in X$ by $F$.
Note that $(X',\Delta')$ is a klt pair.
Let $E'$ be the valuation induced on $x'\in X'$ by $E$.
By Proposition~\ref{prop:ld-finite-quotient}, we know that $a_{E'}(X',\Delta')<1$.
Hence, by~\cite{BCHM10}, there exists a projective birational morphism $Y'\rightarrow X'$ which only extracts $E'$ over $x'$.
Let $Y$ be the normalization of the main component of $Y'\times_{X'}X$.
Then, the projective birational morphism
$Y\rightarrow X$ is $F$-equivariant.
We have that $E$ is a divisor on $Y$.
By construction, the $F$-orbit of $E$ on $Y$ is the exceptional locus.
Hence, $a_{E'}(X,x)=a<1$ for each prime exceptional divisor $E'$ on $Y$ over $X$.
Thus, $Y\rightarrow X$ is a projective birational morphism only extracting non-terminal valuations from a toric singularity.
We conclude that $Y\rightarrow X$ is a toric morphism.
\end{proof}

\section{Geometric Jordan property for affine toric varieties}

In this section, we prove a geometric Jordan property for affine toric varieties.

\begin{proof}[Proof of Theorem~\ref{introthm:affine-geom-jordan}]

Let $X_0$ be an affine toric variety
and $x_0$ its unique fixed point.
Let $G$ be a finite subgroup of ${\rm Aut}(X_0,x_0)$.
By~\cite{BFMS20}, we know that there exists a normal abelian subgroup $A\leqslant G$ of rank at most $n$ and index at most $c(n)$.
Here, $c(n)$ is a constant only depending on $n=\dim(X)$. It suffices to find a bounded subgroup of $A$ whose image on ${\rm Aut}(\widehat{X}_0,x_0)$
is contained in a maximal torus.

By Lemma~\ref{lem:diag}, 
we may find a formally equivalent model
$A\acts (X,x)$
of $A\acts (X_0,x_0)$
so that $A\leqslant G(X,x)\leqslant {\rm Aut}(X,x)$.
Here, $G(X,x)$ 
is a linear algebraic group
whose connected component contains a maximal torus of $X$.
Recall that since both actions are formally equivalent, we have that
\[
A\acts 
{\rm Aut}(\widehat{X},x)\simeq {\rm Aut}(\widehat{X}_0,x_0) \actsr A.
\]
Thus, if we prove an statement for ${\rm Aut}(\widehat{X},x)$,
then the analogous statement will hold for the formal completion of $X_0$ at $x_0$.
From now on, we will work on the affine toric variety $x\in X$
and consider the groups
$A\leqslant G(X,x)\leqslant {\rm Aut}(X,x)$.
We denote by $G^0(X,x)$ the connected component of $G(X,x)$.

By Lemma~\ref{lem:control-G/G^0}, we know that there exists a constant $k(n)$,
only depending on the dimension $n$, so that
any element of $G(X,x)/G^0(X,x)$ has order at most $k(n)$.
Consider the subgroup
\[
A^{k(n)!}:=\{ x^{k(n)!} \mid x\in A\} \leqslant A.
\]
The image of every element
of $A^{k(n)!}$ in $G(X,x)/G^0(X,x)$ is the identity.
Hence, we conclude that
$A^{k(n)!}$ is contained in $G^0(X,x)$.
Note that the index of $A^{k(n)!}$ in $A$ is bounded by
$k(n)!^n$.
Replacing $A$ with $A^{k(n)!}$, 
we may assume that $A\leqslant G^0(X,x)$.
By Lemma~\ref{lem:action-aut-0}, we may assume that
$A$ acts as the identity on the characteristic group ${\rm Cl}(X)$.

By Proposition~\ref{prop:lifting-aut}, we have an exact sequence
\begin{equation}\label{eq:ses-aut} 
1\rightarrow
\mathbb{T}_{{\rm Cl}(X)}
\rightarrow 
\widetilde{G}(X,x)
\rightarrow 
G^0(X,x)
\rightarrow 
1.
\end{equation} 
Let $a_1,\dots,a_k$ be the generators of $A$
and $\tilde{a}_1,\dots,\tilde{a}_k\in 
{\rm Aut}(\mathbb{A}^{\rho+n},0)$ be liftings to $\mathbb{A}^{\rho+n}$ which lie
in $\widetilde{G}(X,x)$.
We know that 
each $\tilde{a}_i$ commutes with the characteristic quasi-torus.
Let $\tilde{A}$ be the pre-image of $A$
in $\widetilde{G}(X,x)$.
Note that $\tilde{A}$ is a reductive group generated by $\mathbb{T}_{{\rm Cl}(X)}$ and the $\tilde{a}_i$'s.
By the diagram~\eqref{diag}, we can assume that
$\tilde{a}_1,\dots,\tilde{a}_k$
are elements of ${\rm GL}_{\rho+n}(\kk)$.

We can write $\mathbb{T}_{{\rm Cl}(X)}\simeq \mathbb{G}_m^\rho \times A_{{\rm Cl}^0(X)}$.
Let $t\in \mathbb{T}_{{\rm Cl}(X)}$ be either a generator
of $\mathbb{G}_m^\rho$ or of $A_{{\rm Cl}^0(X)}$.
Then, we can write
\[
t\cdot (x_1,\dots,x_{\rho+n})=
(t^{w_1}x_1,\dots,t^{w_{\rho+n}}x_{\rho+n}),
\]
for certain integers $w_1,\dots,w_{\rho+n}$.
Up to reordering $x_1,\dots,x_{\rho+n}$,
we may assume that
\[
w_1=\dots=w_{s_1}<w_{s_1+1}=\dots=w_{s_2}<\dots<
w_{s_{l-1}+1}=\dots=w_{s_l},
\]
where $(s_1,\dots,s_l)$ is a partition of $\rho+n$.
Note that $\tilde{a}_i$ commutes with $t$ if and only if
the corresponding element of ${\rm GL}_{\rho+n}(\kk)$ is a diagonal
square block matrix of size $(s_1,\dots,s_l)$.
Thus, each generator of $\mathbb{T}_{{\rm Cl}(X)}$ gives a partition of $\rho+n$.
We replace $(s_1,\dots,s_l)$ with the coarsest partition
which refine all such partitions.
We consider the torus $\mathbb{T}_d:=\mathbb{G}_m^l$ acting diagonally 
on each set of variables of the partition, i.e., 
\[
(t_1,\dots, t_l)\cdot (x_1,\dots,x_{\rho+n}) =
(t_1x_1,\dots,t_1x_{s_1},
t_2x_{s_1+1},\dots,t_2x_{s_2},\dots,t_lx_{s_l}).
\]
The partition $(s_1,\dots,s_l)$ satisfies that
every element $\tilde{a}_i\in {\rm GL}_{\rho+n}(\kk)$ is a diagonal square block matrix of size $(s_1,\dots,s_l)$ and
$\mathbb{T}_{{\rm Cl}(X)}\leqslant \mathbb{T}_d$.
Note that $\sum_{j=1}^l (s_j-1)\leq n$.
Indeed, we have that
\[
n=\dim(X)=
\dim(\mathbb{A}^{\rho+n}/\mathbb{T}_{{\rm Cl}(X)})
\geq 
\dim(\mathbb{A}^{\rho+n}/\mathbb{T}_d)\geq
\sum_{j=1}^l (s_j-1)=:s.
\]
We conclude that
$\tilde{A}\leqslant \bigoplus_{j=1}^l {\rm GL}_{s_j}(\kk)$
with $\sum_{j=1}^l(s_j-1)\leq n$.
Note that $\tilde{A}$ may not be a finite group.
However, by the exact sequence~\eqref{eq:ses-aut}, we have that its image 
\begin{equation}\label{eq:pgl}
\tilde{A}_P \leqslant 
\bigoplus_{j=1}^l{\rm PGL}_{s_j}(\kk)
\end{equation} 
is a finite abelian subgroup of rank at most $n$.
Note that the group on~\eqref{eq:pgl}
is only supported on those 
$s_j$ which are larger than one.
Since $\sum_{j=1}^l (s_j-1)\leq n$, we 
can apply the Jordan property for ${\rm PGL}$.
We conclude that there exists a subgroup 
\[
\tilde{H}_P \leqslant \tilde{A}_P \text{ and } \tilde{f}_P \in \bigoplus_{j=1}^l{\rm PGL}_{s_j}(\kk)
\]
so that 
$\tilde{f}^{-1}_P \tilde{H}_P\tilde{f}_P \leqslant \mathbb{G}_m^s$,
where $\mathbb{G}_m^s$ is a maximal torus of the product of projective general linear groups.
Furthermore, the index of $\tilde{H}_P\leqslant \tilde{A}_P$
is bounded by $e(n)$.
We denote by 
\[
\tilde{f} \in \bigoplus_{j=1}^l {\rm GL}_{s_j}(\kk)\leqslant 
{\rm GL}_{\rho+n}(\kk)
\]
a lifting of $\tilde{f}_P$.
Let $\tilde{H}$ be the pre-image of $H_P$ on ${\rm GL}_{\rho+n}(\kk)$.
Then, we have that 
$\tilde{f}^{-1}\tilde{H}\tilde{f}\leqslant \mathbb{G}_m^{\rho+n}$,
where $\tilde{H}\leqslant \tilde{A}$ is a subgroup of index at most $e(n)$.
Let $f$ be the image of $\tilde{f}$ in ${\rm Aut}(X,x)$.
Let $H$ be the image of $\tilde{H}$ in 
${\rm Aut}(X,x)$.
Note that $H\leqslant A$ has index at most $e(n)$
and 
\[
f^{-1}Hf \leqslant \mathbb{T}_X
\leqslant {\rm Aut}(X,x)
\]
factors through a maximal torus of the affine toric variety.
Since $H$ is a subgroup of $A$, then its rank is at most $n$ as claimed.
Then, we have that 
\[
f^{-1}Hf \leqslant \mathbb{T}_X \leqslant {\rm Aut}(\widehat{X},x) \simeq {\rm Aut}(\widehat{X}_0,x_0),
\]
where by abuse of notation we are denoting $f$ and $H$ their images in ${\rm Aut}(\widehat{X},x)$.
We conclude that the image of $H$ in ${\rm Aut}(\widetilde{X}_0,x_0)$
conjugates to a maximal torus.
\end{proof}

\begin{remark}\label{rem:com}{\em  
In the notation of the proof of Theorem~\ref{introthm:affine-geom-jordan}.
If the group $G$ commutes with a subtorus $\mathbb{T}$ of $\mathbb{T}_X$,
then we can choose $f$ so that it commutes with $\mathbb{T}$ as well.}
\end{remark}

\begin{proof}[Proof of Corollary~\ref{introcor:proj-geom-jordan}]
Let $G$ be a finite group acting on a projective toric variety $X$ of dimension $n$.
Let $H$ be an ample $G$-invariant Cartier divisor on $X$.
For each $g\in G$ and $f\in H^0(X,\mathcal{O}_X(H))$,
we have that $g^*f\in H^0(X,\mathcal{O}_X(H))$.
Hence, $G$ acts on the affine variety
\[
C_X:={\rm Spec}\left( \oplus_{u\geq 0}H^0(X,\mathcal{O}_X(uH))\chi^u \right).
\]
Note that $C_X$ is an affine cone singularity and $G$ acts fixing the vertex.
We denote the vertex by $c$.
We denote by $\mathbb{T}_0$ the one-dimensional torus action on $C_X$.
Since $H$ is a Cartier divisor on a projective toric variety,
we can write $H\sim T_X$, 
where $T_X$ is a torus invariant divisor.
Note that we are not assuming $T_X$ is $G$-invariant.
Hence, 
\[
C_X\simeq {\rm Spec}\left( \oplus_{u\geq 0}H^0(X,\mathcal{O}_X(uT_X))\chi^u \right)
\]
is an affine toric variety.
Thus, $G$ acts on a $(n+1)$-dimensional affine toric variety.
The torus $\mathbb{T}_0\times\mathbb{T}_X$ is a maximal torus acting on $C_X$.
Furthermore, the group $G$ commutes with $\mathbb{T}_0$.
By Theorem~\ref{introthm:affine-geom-jordan}, there exists a constant $N(n+1)$, 
so that $G$ admits a normal abelian subgroup $A\leqslant G$ of index at most $N(n+1)$
and
\[
f^{-1}Af \leqslant \mathbb{T}_0\times \mathbb{T}_X 
\leqslant {\rm Aut}(\widehat{C}_X,c).
\]
By remark~\ref{rem:com}, we may assume that $f$ commutes with $\mathbb{T}_0$.
Thus, $f$ descends to an automorphism $f_X$ of $X$.
We conclude that $f_X^{-1}Af_X\leqslant \mathbb{T}_X\leqslant {\rm Aut}(X)$
as claimed.
\end{proof}

\section{Small quotient minimal log discrepancies}
In this section, we prove the main theorem
of this article concerning minimal log discrepancies
of toric quotient singularities.

\begin{proof}[Proof of Theorem~\ref{introthm:quotient-diagram}]
Let $x\in (X,B)$ be a $n$-dimensional toric quotient pair.
By definition,
there exists a Galois morphism 
$\pi_Z\colon Z\rightarrow X$ so that
\[
\pi_Z^*(K_X+B)=K_Z.
\]
Here, $Z$ is an affine toric variety of dimension $n$.
Let $z$ be the pre-image of $x$ in $Z$.
Let $G$ be the subgroup of ${\rm Aut}(Z,z)$
so that $X=Z/G$.
By Theorem~\ref{introthm:affine-geom-jordan}, 
we know that there exists 
a normal abelian subgroup $A\leqslant G$ with the following properties:
$A$ has index at most $N(n)$,
rank at most $n$,
and $A<\mathbb{T}_X \leqslant{\rm Aut}(\widehat{X},x)$.
Let $N:=G/A$ and $Y_0:=Z/A$.
Let $y_0$ be the image of $z$ on $y_0$.
By Proposition~\ref{prop:quotient-toric}, 
we conclude that $Y_0$ is formally toric at $y_0$.
Note that $N$ acts on $Y_0$ and fixes $y_0$.
Furthermore, $X$ is the quotient of $Y_0$ by $N$.
We denote by
$\pi_{Y_0}\colon Z\rightarrow Y_0$
and
$\pi_X\colon Y_0\rightarrow X$
the quotient morphisms.
By Proposition~\ref{prop:log-quotient},
we can find a log pair structure $(Y_0,B_{Y_0})$ so that
\[
K_Z=\pi^*_Y(K_{Y_0}+B_{Y_0})
\text{ and }
K_{Y_0}+B_{Y_0}=\pi^*_X(K_X+B).
\]
By Proposition~\ref{prop:quotient-toric},
we know that the log pair $(Y_0,B_{Y_0})$ is toric
with standard coefficients.
In particular, it is a klt pair (Proposition~\ref{prop:toric-implies-klt}).
By construction, 
the degree of the morphism $Y_0\rightarrow X$ is bounded by $N(n)$.
Indeed, this is an upper bound for the order of $N$.\\

Assume that ${\rm mld}_x(X,B)<N(n)^{-1}$.
We may assume that $N(n)>1$ so that $N(n)^{-1}$.
In particular, the minimal log discrepancy is non-terminal.
Let $\phi\colon X'\rightarrow X$ 
be a projective birational morphism that
extracts a divisor computing the minimal log discrepancy.
The existence of this morphism follows from~\cite{BCHM10}*{Corollary 1.4.3}.
Let $Y'_0$ be the normalization
of the main component of $X'\times_X Y_0$.
Then, we have a commutative diagram as follows
\[
\xymatrix{
Y'_0 \ar[d]_-{\phi_{Y_0}}\ar[r]^-{\pi_{X'}} & X'\ar[d]_-{\phi} \\
(Y_0,B_{Y_0})\ar[r]^-{\pi_X} & (X,B). 
}
\]
Here, the vertical arrows are projective birational morphisms
and the horizontal arrows are Galois covers.
Let $B_{X'}$ be the strict transform of $B$ on $X'$
and $B_{Y'_0}$ be the strict transform of $B_{Y_0}$ on $Y'_0$.
Let $E_{X'}$ be the unique prime exceptional divisor extracted by $\phi$.
Let $F_1,\dots,F_k$ be the prime
exceptional divisors extracted by $\phi_{Y_0}$.
Note that ${\pi_{X'}}_* F_i=E_{X'}$ for each
$i\in \{1,\dots,k\}$.
We can write
\begin{equation} \label{eq:first}
K_{X'}+B_{X'}+(1-{\rm mld}_x(X,B))E_{X'} = \phi^*(K_X+B).
\end{equation} 
By Proposition~\ref{prop:pull-back}
and equality~\eqref{eq:first},
we have that
\begin{align}\label{eq:long-eq}
\pi_{Y_0}^*(K_{Y_0}+B_{Y_0})&= 
\pi_{X'}^*\phi^*(K_X+B)=
\pi^*_{X'}(K_{X'}+B_{X'}+(1-{\rm mld}_x(X,B))E_{X'})=\\ \nonumber 
& K_{Y'_0}+B_{Y'_0}+\sum_{i=1}^k(1-r{\rm mld}_x(X,B))F_i. 
\end{align}
Here, $r$ is the ramification index of $Y'_0\rightarrow X'$
at the generic point of $E_{X'}$.
The last equality in~\eqref{eq:long-eq}, follows from the formula of log discrepancies
under finite quotients (Proposition~\ref{prop:ld-finite-quotient}).
Since the morphism $Y'_0\rightarrow X'$ has degree
at most $N$, we conclude that $r\leq N$.
We conclude that $a_{F_i}(Y,B_Y)<1$ for each 
$i\in \{1,\dots,k\}$.
In particular, $\phi_{Y_0} \colon Y'_0 \rightarrow Y_0$ is a projective birational morphism
which only extracts non-terminal places 
of the toric pair $(Y_0,B_{Y_0})$.
By Proposition~\ref{prop:non-terminal-toric}, we conclude that
$\phi_{Y_0}\colon Y_0'\rightarrow Y_0$ is a projective toric morphism.\\

Let $N_0\leqslant N$ be the maximal normal subgroup
which acts on $F_1\cup\dots\cup F_k$ as the identity.
By Proposition~\ref{prop:normal-subgroup-identity},
we know that $N_0< \mathbb{G}_m^n \leq {\rm Aut}(\widehat{Y}_0,B_{Y_0})$.
Let $(Y,B_Y)$ be the log quotient of $(Y_0,B_{Y_0})$ by $N_0$.
By Proposition~\ref{prop:quotient-toric},
we know that $(Y,B_Y)$ is a formally toric pair at $y$.
Let $N_1:=N/N_0$.
We have a commutative diagram as follows:
\[
\xymatrix{
Y'_0 \ar[d]_-{\phi_{Y_0}}\ar[r]^-{/N_0} & Y'\ar[d]_-{\phi_Y}\ar[r]^-{/N_1} & X'\ar[d]_-{\phi} \\
(Y_0,B_{Y_0})\ar[r] & (Y,B_Y)\ar[r] & (X,B). 
}
\]
Let $G_i$ be the image of $F_i$ on $Y'$
for each $i\in \{1,\dots, k\}$.
By construction, we have that
the action of $N_1$ on
$G_1\cup\dots\cup G_k$
is unramified in codimension one.
In particular, the ramification index of 
$Y'\rightarrow X'$ over $E_{X'}$ is one.
Hence, for any $i$, we have that
\[
a_{G_i}(Y,B_Y)={\rm mld}_x(X,B).
\]
We claim that any such $G_i$ computes
the minimal log discrepancy of $(Y,B_Y)$ at $y$.
Assume this is not the case.
Then, there exists a prime divisor $G'$ over $Y$
with center $y\in Y$ for which
\[
a_G(Y,B_Y)< a_{G_i}(Y,B_Y)={\rm mld}_x(X,B)<1.
\]
By Proposition~\ref{prop:N-equiv-extraction},
there exists a $N_1$-equivariant projective birational morphism 
$Y''\rightarrow (Y,B_Y)$
that extracts $G'$ and its
exceptional divisor is the $N_1$-orbit of $G'$.
Let $X''$ be the quotient of $Y''$ by $N_1$.
Then, we have a diagram as follows
\[
\xymatrix{
Y'' \ar[d]_-{\psi_Y}\ar[r]^-{/N_1} & X''\ar[d]_-{\psi} \\
(Y_0,B_{Y_0})\ar[r]^-{/N_1} & (X,B). 
}
\]
By construction,
the projective birational morphism
$\psi\colon X''\rightarrow X$
extracts a unique prime divisor $E_{X''}$.
By Proposition~\ref{prop:ld-finite-quotient}, we have that
\[
a_{E_{X''}}(X,B)=
\frac{a_G(Y,B_Y)}{r_0}.
\]
Here, $r_0$ is the ramification
index of $Y''\rightarrow X''$ at the generic point of $E_{X''}$.
Since $r_0\geq 1$, we conclude that
\[
a_{E_{X''}}(X,B)=
\frac{a_G(Y,B_Y)}{r_0}<
a_{G}(Y,B_Y)< {\rm mld}_x(X,B).
\]
Note that $E_{X''}$ is a prime divisor whose image on $X$ equals $x$.
Hence, the strict inequality
$a_{E_{X''}}(X,B)<{\rm mld}_x(X,B)$
gives us a contradiction.
We conclude that
\[
{\rm mld}_y(Y,B_Y)=
a_{G_i}(Y,B_Y)=
a_{E_{X'}}(X,B)=
{\rm mld}_x(X,B).
\]
Note that the affine toric variety
$(Y,B_Y)$
together with the Galois morphisms
$Z\rightarrow Y\rightarrow X$
satisfies all the conditions of the theorem.
This concludes the proof.
\end{proof}

\begin{proof}[Proof of Theorem~\ref{introthm:small-quotient-mlds}] 
Let $G<{\rm GL}_n(\kk)$ be a finite subgroup
and $X:=\kk^n/G$.
Let $Q\leqslant G$ be the normal subgroup
generated by quasi-reflections.
We let $G_0:=G/Q$.
By Chevalley-Shepard-Todd Theorem, we know that 
$\kk^n/Q\simeq \kk^n$.
Furthermore, the action of $G_0$ in $\kk^n$ is unramified in codimension one.
Hence, we can write 
$X\simeq \kk^n/G_0$
and $\pi\colon \kk^n \rightarrow X$
for the quotient morphism.
We have that
$\pi^*K_X=K_{\kk^n}$.
Let
$\epsilon_n$ as in the statement
of Theorem~\ref{introthm:small-quotient-mlds}
and consider the set
${\rm Qmld}_n\cap (0,\epsilon_n)$.
Let 
\[
{\rm Tmld}_n:= 
\{ {\rm mld}_t(T) \mid 
\text{ 
$t\in T$ is a toric singularity
and $\dim(T)=n$}
\}.
\]
By~\cite{Amb06}, we know that the set
${\rm Tmld}_n\cap (0,\epsilon_n)$
satisfies the ascending chain condition.
Indeed, the set ${\rm Tmld}_n$ satisfies the ascending chain condition.
Note that we have a natural inclusion 
\[ 
{\rm Tmld}_n(T) \cap (0,\epsilon_n)
\subset  
{\rm Qmld}_n(T)\cap  (0,\epsilon_n).
\]
Indeed, every toric minimal log discrepancy is computed by a $\qq$-factorial toric minimal log discrepancy,
which is, in particular, a quotient
minimal log discrepancy.
By Theorem~\ref{introthm:small-quotient-mlds}, we know that there are Galois quotients
\[
\kk^n\rightarrow T \rightarrow X, 
\]
so that $T$ is an affine toric singularity
and
\[
{\rm mld}_t(T)={\rm mld}_x(X),
\]
where $t$ is the image of $\{0\} \in \kk^n$ in $T$.
In particular, we conclude that the equality 
\begin{equation}\label{eq:equality-tmld-qmld} 
{\rm Tmld}_n\cap (0,\epsilon_n) 
= 
{\rm Qmld}_n\cap (0,\epsilon_n)
\end{equation} 
holds.
Hence, the set
${\rm Qmld}_n\cap (0,\epsilon_n)$
satisfies the ascending chain condition.

Now, we turn to prove the statement about accumulation points.
Let ${\rm Tmld}_{n,S}$ be the set of
${\rm mld}_t(T,\Delta_T)$ so that
$(T,\Delta_T)$ is a $n$-dimensional toric pair, $\Delta_T$ has standard coefficients, and $t\in T$ is a closed point.
By~\cite{Amb06}*{Theorem 1.(3)}, we know that 
\[
{\rm Acc}\left(
{\rm Tmld}_n 
\right) 
\subset 
\{0\} 
\cup 
\bigcup_{1\leq d\leq n-1}
{\rm Tmld}_{d,S}.
\] 
Here, ${\rm Acc}$ stands for the accumulation points of the set.
Then, by equality~\ref{eq:equality-tmld-qmld}, we obtain that
\[
{\rm Acc}\left({\rm Qmld}_n\cap(0,\epsilon_n)\right)=
{\rm Acc}\left({\rm 
Tmld}_n\cap(0,\epsilon_n)\right)
\subset
\]
\[
\{0\} 
\cup 
\bigcup_{1\leq d\leq n-1}
{\rm Tmld}_{d,S}\cap (0,\epsilon_n)
\subset 
\{0\} 
\cup 
\bigcup_{1\leq d\leq n-1}
{\rm Qmld}_{d,S}\cap (0,\epsilon_n).
\]
This proves the second claim.
\end{proof}

\begin{proof}[Proof of Corollary~\ref{introcor:index}]
Let $x\in (X,B)$ be a toric quotient singularity of dimension $n$
and $m={\rm mld}_x(X,B)\in (0,\epsilon_n)$.
Here, we consider $\epsilon_n$ as in Theorem~\ref{introthm:quotient-diagram}.
Then, we can find a formally toric pair $(T,B_T)$
and a Galois quotient
$T\rightarrow X$ so that
\begin{itemize}
\item ${\rm mld}_t(T,\Delta_T)={\rm mld}_x(X,B)$
where $t$ is the pre-image of $x$ in $T$, and 
\item the degree of the finite map
$\pi \colon T\rightarrow X$ is bounded above by
$\epsilon_n^{-1}$.
\end{itemize}
By~\cite{Amb09}*{Theorem 2.1}, we can find a constant $I(m,n)$, only depending on $m$ and $n$, so that
\[
I(m,n)(K_T+\Delta_T)\sim 0.
\]
We conclude that 
\[
\deg(\pi)I(m,n)(K_X+B)\sim 0,
\]
formally at $x\in X$.
Note that 
$\deg(\pi)I(m,n)\leq I(m,n)\epsilon_n^{-1}$
is bounded by a constant
which only depends on $m$ and $n$.
We conclude that the Cartier index
of $K_X+B$, formally at $x\in X$, 
takes values on a finite set
$\mathcal{F}(m,n)$
which only depends on $m$ and $n$.
\end{proof}

\end{document}